\documentclass{article}

\usepackage{amsfonts, amsmath, amssymb, amsthm, amscd}
\usepackage{ascmac}
\usepackage[dvipdfm]{color}
\usepackage{verbatim}
\usepackage[dvips]{graphicx}
\usepackage{graphics}
\usepackage[all,2cell]{xypic}

\objectmargin+{1mm}
\labelmargin+{0.8mm}
\SelectTips{cm}{12}

\UseAllTwocells

\theoremstyle{plain}  
\newtheorem{thm}{Theorem}[section]

\newtheorem{lem}[thm]{Lemma}
\newtheorem{prop}[thm]{Proposition}

\theoremstyle{definition}

\newtheorem{para}[thm]{}

\theoremstyle{remark}

\DeclareMathOperator{\cA}{\mathcal{A}}
\DeclareMathOperator{\cB}{\mathcal{B}}
\DeclareMathOperator{\cC}{\mathcal{C}}
\DeclareMathOperator{\calD}{\mathcal{D}}
\DeclareMathOperator{\cE}{\mathcal{E}}
\DeclareMathOperator{\cF}{\mathcal{F}}

\DeclareMathOperator{\cT}{\mathcal{T}}

\DeclareMathOperator{\bE}{\mathbf{E}}

\DeclareMathOperator{\bbK}{\mathbb{K}}

\UseAllTwocells

\def\sn{\smallskip\noindent}

\newcommand{\cf}{\textrm{cf.}\;}
\newcommand{\Ch}{\operatorname{\bf Ch}}

\newcommand{\Coker}{\operatorname{Coker}}
\newcommand{\Cone}{\operatorname{Cone}}

\newcommand{\Homo}{\operatorname{H}}

\newcommand{\id}{\operatorname{id}}

\newcommand{\isoto}{\overset{\scriptstyle{\sim}}{\to}}

\newcommand{\onto}[1]{\stackrel{#1}{\to}}
\newcommand{\op}{\operatorname{op}}

\newcommand{\qis}{\operatorname{qis}}

\newcommand{\rdef}{\twoheadrightarrow}
\newcommand{\RelEx}{\operatorname{\bf RelEx}}
\newcommand{\rinc}{\hookrightarrow}

\title{Negative $K$-groups of abelian categories}
\date{}

\author{Satoshi Mochizuki}

\begin{document}

\maketitle

\begin{abstract}
We prove that negative $K$-groups of small abelian categories 
are trivial.
\end{abstract}

\section*{Introduction}
\label{sec:Intro}

\sn
In the celebrated paper \cite{Sch06}, 
Schlichting predicted that for any small abelian 
category $\cA$ and any positive integer $n$, 
the negative $K$-group $\bbK_{-n}\cA$ 
is trivial. 
He proved that 
it is true for $n=1$ and any $\cA$ and 
for any $n$ and any noetherian abelian category $\cA$. 
The goal of this paper, we will prove this conjecture 
for any $n$ and any $\cA$. 
We use the technique established in \cite{Moc13} and \cite{HM13}. 
The main point is making use of higher derived categories $\calD^n\cA$ 
of $\cA$. 
(For definition, see \ref{para:higehr derived categories}.) 
In \cite{HM13}, we give an interpretation of negative 
$K$-groups as the obstruction groups of idempotent completeness 
of higher derived categories. 
(See \ref{para:higehr derived categories} $\mathrm{(2)}$.) 
The main theorem is the following.

\begin{thm}
\label{thm:S conj}
For any abelian category $\cA$ and for any 
positive integer $n\geq 2$, 
$\calD^n\cA$ is trivial. 
In particular, if $\cA$ is essentially small, 
the $-m$-th $K$-group
$\bbK_{-m}\cA$ of $\cA$ is trivial for any $m>0$. 
\end{thm}

\sn
In section 1, we review several notions about relative exact categories 
established in \cite{Moc13} and \cite{HM13}. 
In section 2, we give a proof of the main theorem.

\section{Relative exact categories}
\label{sec:rel exact cat}

\sn
In this section, we recall the fundamental notions and properties 
for relative exact categories from \cite{HM13} and \cite{Moc13}.

\begin{para}[\bf Relative exact categories]
\label{para:ext axiom}
$\mathrm{(1)}$ 
A {\it relative exact category} $\bE=(\cE,w)$ is a pair of 
an exact category $\cE$ 
with a specific zero object $0$ 
and a class of morphisms in $\cE$ 
which is closed under finite compositions. 
Namely $w$ satisfies the following two axioms.\\
{\bf (Identity axiom).} 
For any object $x$ in $\cE$, 
the identity morphism $\id_x$ is in $w$.\\
{\bf (Composition closed axiom).} 
For any composable morphisms $\bullet \onto{a} \bullet \onto{b} \bullet$ 
in $\cE$, if $a$ and $b$ are in $w$, then $ba$ is also in $w$.\\
For any relative exact category $\bE$, 
we write $\cE_{\bE}$ and $w_{\bE}$ 
for the underlying exact category and 
the class of morphisms of $\bE$ 
respectively.\\
$\mathrm{(2)}$ 
A {\it relative exact functor} between relative exact categories 
$f:\bE=(\cE,w) \to (\cF,v)$ is 
an exact functor $f:\cE \to \cF$ such that 
$f(w)\subset v$ and $f(0)=0$. 
We denote the category of relative exact categories and relative exact functors 
by $\RelEx$.\\
$\mathrm{(3)}$ 
We write $\cE^w$ for the full subcategory of $\cE$ 
consisting of those object $x$ such that the canonical morphism 
$0 \to x$ is in $w$. 
We consider the following axioms.\\
{\bf (Strict axiom).} 
$\cE^w$ is an exact category such that 
the inclusion functor $\cE^w \rinc \cE$ is exact 
and reflects exactness.\\
{\bf (Very strict axiom).} 
$\bE$ satisfies the strict axiom and 
the inclusion functor $\cE^w \rinc \cE$ induces a fully faithful 
functor $\calD_b(\cE^w) \rinc \calD_b(\cE)$ on the bounded derived categories.\\
We denote the category of strict (resp. very strict) relative exact categories 
by $\RelEx_{\operatorname{strict}}$ 
(resp. $\RelEx_{\operatorname{vs}}$).
\end{para}

\begin{para}[\bf Derived category]
\label{para:derivd cat}
We define the {\it derived categories} of 
a strict relative exact category $\bE=(\cE,w)$ by the following formula
$$\calD_{\#}(\bE):=\Coker(\calD_{\#}(\cE^w) \to \calD_{\#}(\cE))$$
where $\# =b$, $\pm$ or nothing. 
Namely $\calD_{\#}(\bE)$ is a Verdier quotient of $\calD_{\#}(\cE)$ 
by the thick subcategory of $\calD_{\#}(\cE)$ 
spanned by the complexes in $\Ch_{\#}(\cE^w)$.
\end{para}

\begin{para}[\bf Quasi-weak equivalences]
\label{para:quasi-weak equiv}
Let 
$P_{\#}:\Ch_{\#}(\cE) \to \calD_{\#}(\bE)$ 
be the canonical quotient functor. 
We denote the pull-back of the class of all isomorphisms in 
$\calD_{\#}(\bE)$ 
by $qw_{\#}$ or simply $qw$. 
We call a morphism in $qw$ a {\it quasi-weak equivalence}. 
We write $\Ch_{\#}(\bE)$ for a pair $(\Ch_{\#}(\cE),qw)$. 
We can prove that 
$\Ch_{\#}(\bE)$ is a complicial biWaldhausen 
category in the sense of \cite[1.2.11]{TT90}. 
In particular, it is a relative exact category. 
The functor $P_{\#}$ induces an equivalence of triangulated categories 
$\cT(\Ch_{\#}(\cE),qw)\isoto\calD_{\#}(\bE)$ 
(See \cite[3.2.17]{Sch11}). 
If $w$ is the class of all isomorphisms in $\cE$, 
then $qw$ is just the class of all quasi-isomorphisms in $\Ch_{\#}(\cE)$ 
and we denote it by $\qis$. 
\end{para}

\begin{para}[\bf Consistent axiom]
\label{para:consis axiom}
Let $\bE=(\cE,w)$ be a strict relative exact category. 
There exists the canonical functor $\iota^{\cE}_{\#}:\cE \to \Ch_{\#}(\cE)$ 
where $\iota^{\cE}_{\#}(x)^k$ is $x$ if $k=0$ and $0$ if $k\neq 0$. 
We say that $w$ (or $\bE$) satisfies 
the {\it consistent axiom} 
if $\iota^{\cE}_b(w)\subset qw$. 
We denote the full subcategory of 
consistent relative exact categories in $\RelEx$ by 
$\RelEx_{\operatorname{consist}}$. 
\end{para}

\begin{para}[\bf Higher derived categories]
\label{para:higehr derived categories}
(\cf \cite[3.1, 3.2]{HM13}). 
Let $\bE$ be a very strict consistent relative exact category and 
we denote $n$-th times iteration of $\Ch$ for $\bE$ 
by $\Sigma^{n}\bE$ and 
$\calD^n(\bE):=\calD_b(\Sigma^{n}\bE)$ 
the {\it{$n$-th higher derived category}} of $\bE$. 
Then for any positive integer $n$, we have\\
$\mathrm{(1)}$ 
$\bbK_{-n}(\bE) \simeq \bbK_0(\calD^n(\bE))$.\\
$\mathrm{(2)}$ 
$\bbK_{-n}(\bE)$ is trivial if and only if 
$\calD^n(\bE)$ is idempotent complete.\\
$\mathrm{(3)}$ 
The canonical functor $\Sigma \bE \to \Ch_b\Sigma\bE$ 
induces an equivalence of triangulated categories 
$\calD\bE\isoto\calD_b\Sigma\bE$.
\end{para}

\section{Proof of the main theorem}
\label{sec:vani for abel cat}

\sn
In this section, 
we prove Theorem~\ref{thm:S conj}. 
Let $\cA$ be an essentially small abelian category 
and $n$ a positive integer $n\geq 2$. 

\begin{para}
\label{para:setting of B}
We have the equalities 
\begin{equation}
\label{equ:higher derived cat}
\calD_n\cA\underset{\textbf{I}}{=}
\calD_b\Sigma^n\cA\underset{\textbf{II}}{\isoto}
\calD\Sigma^{n-1}\cA \underset{\textbf{I}}{=}
\calD\Ch^{n-1}\cA/\calD{(\Ch^{n-1}\cA)}^{w_{\Sigma^{n-1}\cA}}.
\end{equation}
Here the equalities \textbf{I} and \textbf{II} just come from definitions 
and \ref{para:higehr derived categories} $\mathrm{(3)}$ 
respectively. 
For simplicity, 
we put $\cB=\Ch^{n-1}\cA$, 
$\cF=\Ch^{n-2}\cA$ and 
$v=w_{\Sigma^{n-1}\cA}$. 
Then the pair $(\cB,v)$ 
is a complicial exact category with weak equivalences 
or a bicomplicial pair in the sense of \cite{Sch11} or \cite{Moc13}. 
Therefore $(\cB,v)$ is very strict by \cite[3.9]{Moc13} 
and hence 
the functor 
$\calD_{\#}\cB^v \to \calD_{\#}\cB$ 
induced from the inclusion functor $\cB^v \rinc \cB$ 
is fully faithful for $\#\in\{b,\pm,\operatorname{nothing}\}$ 
by \cite[1.2]{HM13}. 
\end{para}

\sn
By virtue of equality $\mathrm{(\ref{equ:higher derived cat})}$, 
Theorem~\ref{thm:S conj} follows from 
Proposition~\ref{prop:equiv of tricat} below.

\begin{prop}
\label{prop:equiv of tricat}
The functor 
$\calD_{\#}\cB^v \to \calD_{\#}\cB$ 
induced from the inclusion functor $\cB^v \rinc \cB$ 
is an equivalence of triangulated categories 
for $\#\in\{\pm,\operatorname{nothing}\}$.
\end{prop}

\begin{proof}[\bf Proof]
For $\#=+$, we use lemma \ref{lem:TT90} below. 
We check the condition $\mathrm{(\ast)}$ 
in \ref{lem:TT90} for $\cC=\cB$ and $\calD=\cB^v$. 
For any complex $b$ in $\Ch\cF$, 
there is a canonical epimorphism $\Cone\id_b[-1]\rdef b$ in $\Ch\cF$ 
with $\Cone\id_b[-1]\in \cB^v$. 
Hence we obtain the result for $\#=+$. 
Since we have $\calD_{-}\cB={(\calD_{+}\cB^{\op})}^{\op}$ 
and $\calD_{-}\cB^v={(\calD_{+}{(\cB^{\op})}^{v^{\op}})}^{\op} $, 
we also get the result for $\#=-$. 
Finally notice that $\calD\cB$ 
is generated by $\calD_{\pm}\cB\overset{\sim}{\leftarrow}
\calD_{\pm}\cB^v$. 
Therefore $\calD\cB^v \to \calD\cB$ is essentially surjective 
and we complete the proof.
\end{proof}

\begin{lem}
\label{lem:TT90}
{\rm (\cf \cite[1.9.5, 1.9.7]{TT90})}. 
Let $\cC$ be an abelian category and 
$\calD$ an idempotent complete strict exact subcategory of $\cC$. 
If ``$\calD$ has enough objects to resolve" in the following 
sense $\mathrm{(\ast)}$, 
then the inclusion functor $\Ch_{+}\calD \to \Ch_{+}\cC$ 
induces an equivalence of triangulated categories 
$\calD_{+}\calD\isoto\calD_{+}\cC$.\\
$\mathrm{(\ast)}$ 
For any integer $k$, any complex $x$ in $\Ch_{+}\cC$ 
such that $\Homo_i x=0$ for $i\leq k$ 
and any epimorphism in $\cC$, 
$a\rdef \Homo_{k+1}x$, 
then there exists an object $d$ in $\calD$ 
and a morphism $d\to a$ such that 
the composition is an epimorphism in $\cC$.
\qed
\end{lem}

\end{document}